\documentclass[11pt]{amsart}
\usepackage{amssymb}
\usepackage{amsmath}
\usepackage{graphicx,psfrag}
\usepackage{fancyhdr}
\usepackage[top=1.4in, bottom=1.1in, left=1.3in, right=1.3in]{geometry}

\usepackage{enumerate}

\title{On the random greedy $F$-free hypergraph process}

\date{December 27, 2014}
\author{Daniela K\"uhn, Deryk Osthus and Amelia Taylor}
\thanks{The research leading to these results was partially supported by the European Research Council
under the European Union's Seventh Framework Programme (FP/2007--2013) / ERC Grant
Agreements no. 258345 (D.~K\"uhn) and 306349 (D.~Osthus).}

\newtheorem{firstthm}{Proposition}[section]
\newtheorem{theorem}[firstthm]{Theorem}

\newtheorem{lemma}[firstthm]{Lemma}
\newtheorem{cor}[firstthm]{Corollary}

\newtheorem{conj}[firstthm]{Conjecture}

\newdimen\margin   
\def\textno#1&#2\par{%
   \margin=\hsize
   \advance\margin by -4\parindent
          \setbox1=\hbox{\sl#1}%
   \ifdim\wd1 < \margin
      $$\box1\eqno#2$$%
   \else
      \bigbreak
      \hbox to \hsize{\indent$\vcenter{\advance\hsize by -3\parindent
      \it\noindent#1}\hfil#2$}%
      \bigbreak
   \fi}

\begin{document}

\def\COMMENT#1{}
\def\TASK#1{}

\def\eps{{\varepsilon}}
\newcommand{\ex}{\mathbb{E}}
\newcommand{\pr}{\mathbb{P}}
\newcommand{\cA}{\mathcal{A}}
\newcommand{\cB}{\mathcal{B}}
\newcommand{\cC}{\mathcal{C}}
\newcommand{\cD}{\mathcal{D}}
\newcommand{\cF}{\mathcal{F}}
\newcommand{\cH}{\mathcal{H}}
\newcommand{\cS}{\mathcal{S}}

\begin{abstract}  \noindent
Let $F$ be a strictly $k$-balanced $k$-uniform hypergraph with $e(F)\geq |F|-k+1$ and maximum co-degree at least two.
The random greedy $F$-free process constructs a maximal $F$-free hypergraph as follows.
Consider a random ordering of the hyperedges of the complete $k$-uniform hypergraph $K_n^k$ on $n$ vertices. 
Start with the empty hypergraph on $n$ vertices. Successively consider the hyperedges $e$ of $K_n^k$ in the given ordering, and add $e$ to the
existing hypergraph provided that $e$ does not create a copy of $F$.
We show that asymptotically almost surely this process terminates at a hypergraph with $\tilde{O}(n^{k-(|F|-k)/(e(F)-1)})$ hyperedges. This is best possible up to logarithmic factors.
\end{abstract}

\maketitle

\section{Introduction}\label{sec:intro}

\subsection{Results}
Fix a $k$-uniform hypergraph $F$. In this paper, we study the following random greedy process, which constructs a maximal $F$-free $k$-uniform hypergraph. Assign a birthtime which is uniformly distributed in $[0,1]$ to each hyperedge of the complete $k$-uniform hypergraph $K_n^k$ on $n$ vertices. Start with the empty hypergraph on $n$ vertices at time $p=0$. Increase $p$ and each time that a new hyperedge is born, add it to the hypergraph provided that it does not create a copy of $F$ (edges with equal birthtime are added in any order). Denote the resulting hypergraph at time $p$ by $R_{n,p}$. 

The random greedy graph process (i.e.~the case when $k=2$) has been studied for many graphs. The initial motivation (see for example \cite{ersuwi}) was to study the Ramsey number $R(3,t)$. Indeed, the best current lower bounds on $R(3,t)$ were obtained via the study of the
triangle-free process (\cite{bohkeev2}, \cite{FGM}). Osthus and Taraz \cite{greedygraph} gave an upper bound on the number of edges in the graph $R_{n,1}$ when $F$ is strictly $2$-balanced, showing that a.a.s. $R_{n,1}$ has maximum degree $O(n^{1-(|F|-2)/(e(F)-1)} (\log n)^{1/(\Delta(F)-1)})$. 
(Here a.a.s.~stands for `asymptotically almost surely', i.e.~for the property that an event occurs with probability tending to one as $n$ tends to infinity.) Results for the cases when $F=C_4$ and $F=K_4$ were obtained independently by Bollob\'as and Riordan \cite{bolrio}. Bohman and
Keevash~\cite{bohkeev} showed that a.a.s.~$R_{n,1}$ has minimum degree $\Omega(n^{1-(|F|-2)/(e(F)-1)} (\log n)^{1/(e(F)-1)})$ whenever $F$ is strictly $2$-balanced and conjectured that this gives the correct order of magnitude. Improved upper bounds have been obtained for some graphs. For instance, the number of edges has been determined asymptotically when $F$ is a cycle (\cite{boh}, \cite{bohkeev2}, \cite{FGM}, \cite{pico}, \cite{warn}) and when $F=K_4$ (\cite{warn2}, \cite{wolf}).
Picollelli \cite{pico2} determined asymptotically the number of edges when $F$ is a diamond, i.e.~the graph obtained by removing one edge from $K_4$. Note that this graph is not strictly $2$-balanced.

Much less is known about the process when $F$ is a $k$-uniform hypergraph and $k\geq 3$. The only known upper bound is due to Bohman,
Mubayi and Picollelli \cite{bomupi}, who studied the $F$-free process when $F$ is a $k$-uniform generalisation of a graph triangle
(with an application to certain Ramsey numbers).
In this paper, we obtain a generalisation of the upper bound in \cite{greedygraph} to strictly $k$-balanced hypergraphs. Here we say that a $k$-uniform hypergraph $F$ is \emph{strictly $k$-balanced} if $|F|\geq k+1$ and for all proper subgraphs $F'\subsetneq F$ with $|F'|\geq k+1$ we have
$$\frac{e(F)-1}{|F|-k}> \frac{e(F')-1}{|F'|-k}.$$
We also need the following definition. Given a hypergraph $H$ and $i\in \mathbb{N}$, we define the \emph{maximum i-degree} of $H$ by
$$\Delta_i(H):=\max\{d_H(U):U\subseteq V(H), |U|=i\},$$ where $d_H(U)$ is the number of hyperedges in $H$ containing $U$.

\begin{theorem}\label{thm:main}
Let $k\in \mathbb{N}$ be such that $k\geq 2$. Let $F$ be a strictly $k$-balanced $k$-uniform hypergraph which has $v$ vertices and $h\geq v-k+1$ hyperedges.%
\COMMENT{condition on no. edges ensures that $\frac{v-k}{h-1}\leq 1$, we need this to ensure that $t\geq\Delta_{k-1}(F)$ later on, otherwise $\cS=\emptyset$.}
Suppose $\Delta_{k-1}(F)\geq 2$.%
\COMMENT{we need this for the bound on $\mu_1$ at the beginning of Lemma~\ref{lem:XS}.}
Then there exists a constant $c$ such that a.a.s. 
\begin{equation}\label{eq:main}
\Delta_{k-1}(R_{n,1})< t \hspace{0.5cm} \text{ where } \hspace{0.5cm} t:=cn^{1-\frac{v-k}{h-1}} (\log n)^{\frac{3}{\Delta_{k-1}(F)-1}-\frac{1}{h-1}}.
\end{equation}
In particular, a.a.s. $R_{n,1}$ has at most $tn^{k-1}$ hyperedges.
\end{theorem}
Note that Theorem~\ref{thm:main} applies, for example, to all $k$-uniform cliques $K_v^k$ on $v\geq k+1$ vertices%
\COMMENT{$$\frac{\binom{v}{k}-1}{\binom{v-1}{k}-1}\cdot \frac{v-k-1}{v-k}>\frac{\binom{v}{k}}{\binom{v-1}{k}}\cdot \frac{v-k-1}{v-k}=\frac{v(v-k-1)}{(v-k)^2}\geq 1$$ for all $v\geq k+2$. If $v=k+1$ then $K_v^k$ is trivially strictly
$k$-balanced.}
and more generally to all balanced complete $\ell$-partite $k$-uniform hypergraphs with $\ell\ge k$ and more than $k$ vertices.%
\COMMENT{(AT: changed subscript $V_k$ to $V_\ell$; changed strict inequalities in ``$x\ell(k-1)-k^2>0$ iff $x> \frac{k^2}{\ell(k-1)}$ iff $v'>k^2/(k-1)$, which holds since $v'\ge k+2$'' to ``$\geq$" so true for $k=2$.) Let $F$ be the $\ell$-partite, $k$-uniform hypergraph $K^k_{r,r,\dots,r}$ with $r\geq 2$, $\ell\geq k$, $|F|=r\ell\geq k+1$ (if $r=1$ graph is a clique, see previous case). 
Let $V_1, \dots V_\ell$ denote the vertex classes of $F$. $F$ has $k$-density
$$d_k(F)=\frac{h-1}{v-k}=\frac{r^k\binom{\ell}{k}-1}{r\ell-k}.$$
Let $H$ be any subgraph of $F$ on $k+1 \leq v'<r\ell$ vertices. 
Let $h'$ be the number of edges of $H$.
We claim that $h' \le (v'/\ell)^k \binom{\ell}{k}$.
To see this, let $v_i$ be the number of vertices of $H$ which lie in $V_i$. 
Suppose that the partition $P$ is not equitable, ie $v_1\ge v_2+2$, say.
Let $Q$ be the partition where 1 vertex from the first class is moved to the second one
and let $H^*$ be the resulting graph.
For any $k$-tuple $S$ which includes $1,2$, the number of edges 
induced by the corresponding vertex classes is larger in $H^*$ than in $H$.
For any $k$-tuple $S$ which includes $1$ but not $2$, let $S':=S \cup \{ 2 \} \setminus \{ 1 \}$.
Then the sum of the number of edges induced by $S$ and the number of edges induced by $S'$ in the graph $H$
is equal to the corresponding sum in the graph $H^*$.
Similarly if $S$ contains neither $1$ nor $2$, the number of edges is not affected.
This proves that the number of edges is maximized by an equitable partition.
Considering a fractional version of this argument (ie moving eg half a vertex) implies that
$h' \le (v'/\ell)^k \binom{\ell}{k}$.
Finally note that $\frac{(v'/\ell)^k\binom{\ell}{k}-1}{v'-k}<d_k(F).$
Indeed, to see this suppose first that $v'\ge k+2$. Let $x:=v'/\ell$. It suffices to show that the function
$\frac{x^k\binom{\ell}{k}-1}{x\ell-k}$ is strictly monotone increasing. The numerator of the derivative
is $kx^{k-1}\binom{\ell}{k}(x\ell-k)-\ell(x^k\binom{\ell}{k}-1)=x^{k-1}\binom{\ell}{k}[x\ell(k-1)-k^2]+\ell\geq \ell>0$,
since $x\ell(k-1)-k^2\geq 0$ iff $x\geq \frac{k^2}{\ell(k-1)}$ iff $v'\geq k^2/(k-1)$, which holds since $v'\ge k+2$, $k\geq 2$.
It remains to consider the case when $v'= k+1$. If $\ell>k$ then $h'\le \binom{k+1}{k}=k+1$ and so $d_k(H)=k$.
Thus in this case we need to show that $(r\ell-k)k<r^k\binom{\ell}{k}-1$, which holds (for all $r,k\geq 2$, $k^{1/(k-1)}\leq r \implies rk\leq r^k \implies (r\ell-k)k=rk\ell-k^2<r^k\ell-1\leq r^k\binom{\ell}{k}-1$).
If $\ell=k$ then $h'\le 2$ and so $d_k(H)=1<\frac{r^{k-1}+r^{k-2}+\dots+1}{k}=\frac{r^k-1}{k(r-1)}=d_k(F)$ (holds for any $r\geq 2$).
So $F$ is strictly $k$-balanced. Theorem~\ref{thm:main} applies since $\Delta_{k-1}(F)\ge r\geq 2$ and 
$r^k\binom{\ell}{k} \geq \ell r \geq \ell r -k+1$ (for all $r,k,\ell \geq 2$).
}

Bennett and Bohman \cite{benboh} studied a random greedy independent set algorithm in certain quasi-random hypergraphs. 
Their result can be applied in the context of the $F$-free process to show that if $F$ is a strictly $k$-balanced $k$-uniform hypergraph and every vertex of $F$ lies in at least two hyperedges, then a.a.s. $R_{n,1}$ has $\Omega(n^{k-(|F|-k)/(e(F)-1)} (\log n)^{1/(e(F)-1)})$ hyperedges.%
\COMMENT{see appendix - for statement of theorem and justification of conditions.}
Up to logarithmic factors, this matches the upper bound given in Theorem~\ref{thm:main}.

\subsection{An open question}
There are many natural open questions related to the random greedy $F$-free process.
Here we discuss bounds on the number of edges in $R_{n,1}$ when $F$ is an $\ell$-cycle.
Theorem~\ref{thm:main} applies in the case when $F$ is a $k$-uniform tight cycle.
However, there are other natural notions of a hypergraph cycle: Given $\ell\in \mathbb{N}$ with $\ell <k$, we say that a $k$-uniform hypergraph $C_{\ell,h}$ is an \emph{$\ell$-cycle} of length $h$ if there is a cyclic ordering of its vertices $x_1, \dots, x_{h(k-\ell)}$ and a corresponding ordering on its hyperedges $e_0, \dots, e_{h-1}$ such that $e_i=\{x_{i(k-\ell)+1},\dots, x_{i(k-\ell)+k}\}$. So consecutive hyperedges on the cycle intersect in exactly $\ell$ vertices. The case when $\ell=k-1$ corresponds to $C_{\ell,h}$ being a tight cycle of length~$h$.
It is easy to check that all $\ell$-cycles are strictly $k$-balanced,%
\COMMENT{Let $1\leq \ell\leq k-1$, $v=h(k-\ell)\geq k+1$. Subgraphs with highest $k$-density are paths with $v'=h'(k-\ell)+\ell$ vertices and $h'$ edges. Then $$\frac{h'-1}{v'-k}=\frac{h'-1}{(h'-1)(k-\ell)}=\frac{1}{k-\ell}$$
and
$$\frac{h-1}{(k-\ell)h-k}>\frac{1}{k-\ell} \Leftrightarrow h-1> h-\frac{k}{k-\ell} \Leftrightarrow \frac{k}{k-\ell}>1,$$ so ok.}
but only tight cycles satisfy the co-degree condition in Theorem~\ref{thm:main}.
In the case when $\ell \geq k/2$, $\ell$-cycles meet the conditions in~\cite{benboh}. We conjecture that the bound on the number of hyperedges
in~\cite{benboh} is of the correct magnitude for any $\ell$.

\begin{conj}\label{conj:1}
Let $\ell, k \in \mathbb{N}$ be such that $k\geq 2$ and $k>\ell$ and let $F:=C_{\ell,h}$ be the $\ell$-cycle of length $h$. Then a.a.s. $R_{n,1}$ has $\Theta(n^{\frac{h\ell}{h-1}}(\log n)^\frac{1}{h-1})$ hyperedges.
\end{conj}
One motivation for Conjecture~\ref{conj:1} is that $p=n^{h\ell/(h-1)-k}(\log n)^{1/(h-1)}$ is the threshold for the property that every hyperedge in $H_{n,p}$ lies in an $\ell$-cycle of length $h$.%
\COMMENT{ROUGH idea: When $p=n^{-(v-k)/(h-1)}(\log n)^{1/(h-1)}$, fix some edge $f$ in $H_{n,p}$ and let $X$ be the expected number of copies of $F$ in $H_{n,p}$ containing $f$. $\mu:=\ex(X)\approx n^{v-k}p^{h-1}\approx k^2\log n$. So that, if we assume all events are independent, we could use Chernoff to get
$$\pr[X\leq \mu/2] < e^{-\mu/8}\approx \frac{1}{n^{2k}}.$$ Sum over all possible $k$-sets get that the probability there is an edge in $H_{n,p}$ which is not contained in a copy of $F$ is $o(1)$.}

\subsection{Sketch of the argument}
Rather than studying the random greedy process itself, we are able to prove Theorem~\ref{thm:main} by obtaining precise information
about the random binomial hypergraph $H_{n,p}$. 
(This idea was first used in~\cite{greedygraph}.)
More precisely,  write $H_{n,p}$ for the random binomial $k$-uniform hypergraph on $n$ vertices with hyperedge probability $p$, i.e., each hyperedge is included in $H_{n,p}$ with probability $p$, independently of all other hyperedges. We write $H_{n,p}^-$ for the hypergraph formed by removing all copies of $F$ from $H_{n,p}$. Note that $H_{n,p}$ can also be viewed as the random hypergraph consisting of all hyperedges with birthtime at most $p$. Thus, 
for all $p\in [0,1]$ we have
$$H_{n,p}^- \subseteq R_{n,p} \subseteq R_{n,1}.$$
We will always assume that $K_n^k$, $H_{n,p}$, $H_{n,p}^-$ and $R_{n,p}$ use the vertex set $[n]$.

In Section~\ref{sec:tools}, we collect some large deviation inequalities. The proof of Theorem~\ref{thm:main} is given in Section~\ref{sec:proof}, the strategy is as follows. We first identify the largest point $p$ where we can still use $H_{n,p}$ to approximate the behaviour of $H_{n,p}^-$ (i.e.~for this $p$, only a small proportion of edges of $H_{n,p}$ lie in a copy of $F$). 
Now let $U$ be a set of $k-1$ vertices in $F$ such that $d_F(U)=\Delta_{k-1}(F)$.
Let $\hat{F}$ be the subgraph of $F$ obtained by deleting all those hyperedges which contain~$U$. Let $t$ be as in \eqref{eq:main}. Suppose for a contradiction that there exists a $(k-1)$-set $V$ of degree $t$
in $R_{n,1}$ and let $T$ be the neighbourhood of $V$ in $R_{n,1}$. 
We will show that in this case we would almost certainly find a copy $\alpha$ of
$\hat{F}$ in $H_{n,p}^-[T \cup V]$ which maps $U$ to $V$. 
Since $H_{n,p}^-\subseteq R_{n,1}$, $\alpha$ would also be a copy of $\hat{F}$ in $R_{n,1}[T \cup V]$ which maps $U$ to $V$.
But this actually yields a copy of $F$ in $R_{n,1}$,
a contradiction. So a.a.s. $\Delta_{k-1}(R_{n,1})<t$. It is perhaps surprising that for our analysis the order of hyperedges added after this
critical point $p$ is irrelevant.

\section{Tools}\label{sec:tools}
Let $\cS$ be a collection of subsets of $E(K_n^k)$. For each $\alpha \in \cS$, let $I_\alpha$ denote the indicator variable which equals one if all hyperedges in $\alpha$ lie in $H_{n,p}$ and zero otherwise. Set
$$X:=\sum_{\alpha \in \cS} I_\alpha \hspace{1cm}\text{ and } \hspace{1cm} \mu:=\ex[X].$$

Let $Y$ be the size of a largest hyperedge-disjoint collection of elements of $\cS$ in $H_{n,p}$ (i.e.~the maximum size of a set $\cS'\subseteq \cS$
such that $I_\alpha=1$ for all $\alpha\in \cS'$ and $\alpha\cap \alpha'=\emptyset$ for all distinct $\alpha,\alpha'\in \cS'$).%
  \COMMENT{DK: new bracket}
Erd\H{o}s and Tetali \cite{erdtet} proved the following upper tail bound on $Y$.

\begin{theorem}{\cite{erdtet}.}\label{thm:erdtet}
For every $a\in \mathbb{N}$, we have $\pr[Y\geq a]\leq (e\mu/a)^a$.
\end{theorem}

We also require a lower tail bound on $Y$. For all $\alpha, \alpha' \in \cS$ with $\alpha \neq \alpha'$, we write $\alpha \sim \alpha'$ if $\alpha \cap \alpha'\neq \emptyset$. Define
$$\Delta:= \sum_{\alpha' \sim \alpha} \ex[I_\alpha I_{\alpha'}],$$
where the sum is over all ordered pairs $\alpha' \sim \alpha$ in $\cS$. Also, let%
   \COMMENT{DK: changed $\sum_{\alpha' \sim \alpha}$ to $\sum_{\alpha'\in \cS: \alpha' \sim \alpha}$ below}
$$\eta:=\max_{\alpha \in \cS} \ex[I_\alpha] \hspace{1cm}\text{ and } \hspace{1cm} \nu:=\max_{\alpha \in \cS} \sum_{\alpha'\in \cS: \alpha' \sim \alpha} \ex[I_{\alpha'}].$$
The following bound follows from Lemma 4.2 in Chapter 8 and Theorem A.15 in \cite{probmeth}, see \cite{greedygraph}.

\begin{theorem}\label{thm:janson}
Let $\eps>0$. Then $\pr[Y\leq (1-\eps)\mu]\leq e^{(1-\eps)\mu\nu+\frac{\Delta}{2(1-\eta)}-\frac{\eps^2\mu}{2}}$.
\end{theorem}

\section{Proof of Theorem~\ref{thm:main}}\label{sec:proof}

\subsection{Basic parameters}

Let $F$ be a strictly $k$-balanced $k$-uniform hypergraph which has $v$ vertices and $h$ hyperedges. Let $d:=\Delta_{k-1}(F)$ and choose positive constants $c_1, c_2$ satisfying
$$1/n \ll 1/c_1\ll 1/c_2\ll 1/v,1/h.$$
(Here the notation $a \ll b$ means that we can find an increasing function $f$ for which all of the conditions in the proof are
satisfied whenever $a \leq f(b)$.) Given functions $f$ and $g$, we will write $f=\tilde{O}(g)$ if there exists a constant $c$ such that $f(n)\leq (\log n)^cg(n)$ for all sufficiently large~$n$.

Set%
\COMMENT{In Theorem~\ref{thm:main}, $c=c_1/c_2$.}
$$p:=\frac{1}{c_2(n^{v-k}\log{n})^{1/(h-1)}}\hspace{1cm}\text{ and }\hspace{1cm} t:=c_1np(\log{n})^{3/(d-1)}.$$
Here $p$ is chosen to be as large as possible subject to the constraint that a.a.s.~only a small proportion of the hyperedges of $H_{n,p}$ lie in a copy of $F$.
For each $k+1\leq i \leq v$,%
\COMMENT{note that $F'\neq F$, so it is ok to let $i=v$, just get $h_v=h-1$.}
 we define 
$$h_i:=\max\{e(F'): F'\subsetneq F, |F'|=i\}.$$
Since $F$ is strictly $k$-balanced, we have
$$\frac{h-1}{v-k}>\frac{h_i-1}{i-k}.$$
So for each $k+1\leq i \leq v$ we can define a positive constant
\begin{equation}\label{eq:deltai}
\delta_i:=i-k-(h_i-1)\frac{v-k}{h-1}>0.
\end{equation}
Let%
\COMMENT{AT: changed $<v$ to $\leq v$ in definition of $\delta$.}
$$\delta:=\min\{\delta_i: k+1\leq i \leq v\}.$$
We will often use that for $k+1\leq i \leq v$
\begin{align}\label{eq:qbound}
n^{v-i}p^{h-h_i}\leq  n^{v-i-\frac{v-k}{h-1}(h-h_i)}\stackrel{\eqref{eq:deltai}}{=}n^{v-i-\frac{v-k}{h-1}(h-1-\frac{i-k-\delta_i}{v-k}(h-1))} = n^{-\delta_i}\leq n^{-\delta}.
\end{align}
Note that this bounds the expected number of extensions of a fixed subgraph of $F$ on $i$ vertices into copies of $F$ in $H_{n,p}$.

\subsection{Many copies of $F$ containing a fixed hyperedge}\label{sec:clustering}

For a given hyperedge $f\in E(K_n^k)$, an \emph{$(r,f)$-cluster} is a collection $F_1, F_2, \dots, F_r$ of $r$ copies of $F$ such that each $F_i$ contains $f$ and for each $1< i \leq r$, there exists $f_i\in E(F_i)$ such that $f_i\notin E(F_j)$ for any $j<i$. Define $\cA$ to be the event that $H_{n,p}$ has no $(\log{n},f)$-cluster for any hyperedge $f$. We will bound the probability of $\cA^c$, i.e., the probability that $H_{n,p}$ has a $(\log{n},f)$-cluster for some hyperedge $f$.

\begin{lemma}\label{lem:A}
We have $\pr[\cA^c]\leq n^{-k}$.
\end{lemma}

\begin{proof}
Fix some $f\in E(K_n^k)$. Write $Z_{r,f}$ for the number of $(r,f)$-clusters in $H_{n,p}$, so $Z_{1,f}$ counts copies of $F$ which contain the hyperedge $f$. There are $h$ hyperedges in $F$ which could be mapped to $f$, so
$$\ex[Z_{1,f}]\leq hn^{v-k}p^h\leq e^{-2k}$$
with room to spare.
Let $r<\log n$ and consider a fixed $(r,f)$-cluster $C$ in $H_{n,p}$. Let $Z_C$ be the number of $(1,f)$-clusters in $H_{n,p}$ which contain at least one hyperedge which does not lie in $C$, so each of these $(1,f)$-clusters together with $C$ forms an $(r+1,f)$-cluster. Suppose that $\alpha$ is a $(1,f)$-cluster sharing $k+1\leq i\leq v$ vertices with $C$. The set of hyperedges shared by $\alpha$ and $C$ form a proper subgraph of $F$ on $i$ vertices. Since $F$ is strictly $k$-balanced, $\alpha$ and $C$ can have at most $h_i$ common hyperedges. This allows us to estimate $\ex[Z_C]$ as%
\COMMENT{the $h$ and $v^i$ factors account for different ways (permutations on the vertices) to map the copy of $F$.}
\begin{align*}
\ex[Z_C]&\leq hn^{v-k}p^{h-1}+\sum_{i=k+1}^{v} v^i(rv)^{i-k}n^{v-i}p^{h-h_i}
\stackrel{\eqref{eq:qbound}}{\leq} e^{-3k}+\tilde{O}(n^{-\delta})\leq e^{-2k}.
\end{align*}
If we sum over all $(r,f)$-clusters in $K_n^k$, we find that
\begin{align*}
\ex[Z_{r+1,f}]&\leq \ex[Z_{r,f}]e^{-2k} \leq e^{-2(r+1)k}
\end{align*}
and hence $\ex[Z_{\log n,f}]\leq n^{-2k}$.
By summing over all $f\in E(K_n^k)$, we obtain
\begin{align*}
\pr[\cA^c]\leq \binom{n}{k}n^{-2k}\leq n^{-k},
\end{align*}
as required.
\end{proof}

\subsection{Estimating the number of extensions of a fixed set}\label{sec:estimates}

Recall that $d=\Delta_{k-1}(F)$.%
   \COMMENT{DK reworded this para}
Let $U=\{u_1, u_2, \dots, u_{k-1}\}\subseteq V(F)$ be such that $d_F(U)=d$. Let $N_F(U)$ denote the neighbourhood
of $U$ in~$F$, i.e.~$N_F(U):=\{x\in V(F): U\cup \{x\}\in E(F)\}$. Define $\hat{F}\subseteq F$
which has vertex set $V(F)$ and all hyperedges $f\in E(F)$ such that $|f\cap U|\leq k-2$.
Fix $T\subseteq [n]$ of size $t$ and an ordered sequence $V=(v_1, v_2, \dots, v_{k-1})$ of distinct vertices,
where $v_i \in [n]\setminus T$ for each $1\leq i \leq k-1$. Given a hypergraph $H\subseteq K_n^k$, let $\cS(H)=\cS(H,T,V)$ be the set of all copies of $\hat{F}$ in $H$ such that the following hold:
\begin{itemize}
	\item for each $1\leq i \leq k-1$, $u_i$ is mapped to $v_i$;
	\item $N_F(U)$ is mapped into $T$ and
	\item $V(F)\setminus N_F(U)$ is mapped into $[n]\setminus T$.
\end{itemize}
We let%
   \COMMENT{DK changed $X_{\cS}$ and $X^-_{\cS}$ to $X$ and $X^-$ below and moved def of $\cS:=\cS(K_n^k)$ to beginning of proof.
Also very slightly reworded next para and statement of lemma and made penultimate inequality in (\ref{eq:mu1}) to equality}
$X:=|\cS(H_{n,p})|$ and $X^-:=|\cS(H_{n,p}^-)|$. Note that $X^- \leq X$ since $H_{n,p}^- \subseteq H_{n,p}$.

Note that if $T\subseteq N_{R_{n,1}}(V)$, then $\cS(R_{n,1})=\emptyset$, as otherwise we could find a copy of $F$ in $R_{n,1}$. Since $H_{n,p}^-\subseteq R_{n,1}$, it follows that $X^-=0$. So, in order to prove Theorem~\ref{thm:main}, it will suffice to prove that a.a.s. we have $X^->0$ for any choice of $T, V$. 

\begin{lemma}\label{lem:XS}
Given $T\subseteq [n]$ of size $t$ and an ordered sequence $V=(v_1, v_2, \dots, v_{k-1})$ of distinct vertices,
where $v_i \in [n]\setminus T$ for each $1\leq i \leq k-1$, define $X^-$ as above. Then
$$\pr[(X^-=0) \cap \cA]\leq 2n^{-2t}.$$
\end{lemma}

\begin{proof}
Write $\cS:=\cS(K_n^k)$. Note that%
\COMMENT{$\geq 24h^2t(\log n)^2$ holds for any $c_1^{d-1} \geq 24h^2c_2^{h-1}d^dv^v$.}
\begin{align}\label{eq:mu1}
\mu_1&:=\ex[X] \geq \binom{t}{d}\binom{n-t-k+1}{v-d-k+1}p^{h-d}\geq \frac{tt^{d-1}n^{v-d-k+1}p^{h-d}}{d^dv^v}\nonumber\\
&=\frac{tc_1^{d-1}n^{v-k}p^{h-1}(\log n)^3}{d^dv^v} 
= \frac{c_1^{d-1}}{d^dv^vc_2^{h-1}}t(\log n)^2\geq 24h^2t(\log n)^2.
\end{align}
Let $\cS'(H_{n,p})$ be a hyperedge-disjoint collection of elements of $\cS(H_{n,p})$ of maximum size and let $Y_1:=|\cS'(H_{n,p})|$. In order to apply Theorem~\ref{thm:janson}, we will estimate $\nu$, $\Delta$ and $\eta$.

First we estimate $\nu$. Define
$$\nu^*:=\max_{\alpha \in \cS} \sum_{\alpha'\in \cS: \alpha' \sim \alpha} \ex[I_{\alpha'}\mid I_{\alpha}=1]$$
and note that $\nu\leq \nu^*$. We count the expected number of elements $\alpha' \in \cS(H_{n,p})\setminus \{\alpha\}$ sharing%
   \COMMENT{DK changed $\alpha' \in \cS(H_{n,p})$ to $\alpha' \in \cS(H_{n,p})\setminus \{\alpha\}$. Also once more below.}
at least one hyperedge with some fixed element $\alpha \in \cS$. Note that $\alpha$ and $\alpha'$ must share at least two vertices outside $V$. We let $k+1\leq i+j\leq v$%
\COMMENT{we can have $i+j=v$ if $V(\alpha)=V(\alpha')$ but the edge sets are different.}
 denote the number of shared vertices, where $i$ is the number of vertices shared in $T$. Consider any $\alpha'\in \cS\setminus \{\alpha\}$ sharing $i+j$ vertices with $\alpha$. Let $K$ be the hypergraph on $i+j$ vertices formed by the set of hyperedges shared by $\alpha$ and $\alpha'$. Let $K'$ be the hypergraph on $i+j$ vertices obtained from $K$ by adding all hyperedges of the form $V\cup x$ for each of the $i$ vertices $x\in T$ shared by $\alpha$ and $\alpha'$. Since $K'\subsetneq F$, $e(K')\leq h_{i+j}$ and so $\alpha$ and $\alpha'$ can have at most $h_{i+j}-i$ common hyperedges. Then 
\begin{align*}
\nu &\leq \nu^*\leq \sum_{i+j=k+1}^{v} v^{i+j}t^{d-i}n^{v-d-j}p^{h-d-(h_{i+j}-i)}\\
&=\sum_{i+j=k+1}^{v} v^{i+j}(c_1(\log{n})^{\frac{3}{d-1}})^{d-i}n^{v-(i+j)}p^{h-h_{i+j}} \stackrel{\eqref{eq:qbound}}{=}\tilde{O}(n^{-\delta})=o(1).
\end{align*}
Since $\Delta$ counts the expected number of ordered pairs of elements in $\cS(H_{n,p})$ which share at least one hyperedge, we have
\begin{align*}
\Delta \leq \mu_1\nu^*=o(\mu_1).
\end{align*}
Finally, the probability of a fixed element in $\cS$ being present in $H_{n,p}$ is given by
$$\eta=p^{h-d}=o(1).$$
So we can apply Theorem~\ref{thm:janson} to see that
\begin{align}
\pr[Y_1\leq \mu_1/2]\leq e^{-\mu_1/10}\stackrel{\eqref{eq:mu1}}{\leq} n^{-2t}.
\label{eq:Y1}
\end{align}

We define a \emph{cluster} $(\alpha, F')$ to be the union of an element $\alpha\in \cS'(H_{n,p})$ and a copy $F'$ of $F$ in $H_{n,p}$ which share at least one hyperedge. Note that deleting $F'$ from $H_{n,p}$ to form $H_{n,p}^-$ will destroy $\alpha$. 

We define an auxiliary graph $G$ as follows. For each element of $\cS'(H_{n,p})$ which lies in a cluster, choose one. These clusters form the vertices of $G$. Draw an edge between two vertices in $G$ if the corresponding clusters share a hyperedge. We will use that
\begin{equation}\label{eq:grapheq}
|G|\leq (\Delta(G)+1)\alpha(G)
\end{equation}
(which holds for all graphs) to bound the number of vertices in $G$. We will show that with sufficiently high probability $|G|<Y_1$.
(This in turn implies that at least one element of $\cS'(H_{n,p})$ will remain in $H_{n,p}^-$, i.e.~$X^->0$.)%
    \COMMENT{DK reworded last sentence}

First, we bound $\alpha(G)$. Let $X_2$ be the number of clusters in $H_{n,p}$. We estimate $\mu_2:=\ex[X_2]$, breaking the sum into parts depending on the number $i$ of vertices shared by $\alpha$ and $F'$ in each cluster $(\alpha,F')$. For $k+1\leq i \leq v$,%
\COMMENT{$i=v$ is allowed since $e(\alpha)<e(F)$.}
 we use that $\alpha$ and $F$ intersect in a proper subgraph of $F$ and thus can have at most $h_i$ common hyperedges. The first term in our bound on $\mu_2$ corresponds to those clusters $(\alpha, F')$ where $\alpha$ and $F'$ share exactly one hyperedge:%
\COMMENT{final inequality holds for any $c_2^{h-1}\geq 24e^2h^4$.}
\begin{align}\label{eq:mu2}
\mu_2&=\ex[X_2]\leq \mu_1h^2n^{v-k}p^{h-1}+\sum_{i=k+1}^{v}\mu_1v^in^{v-i}p^{h-h_i}\nonumber\\
&\stackrel{\eqref{eq:qbound}}{\leq} \mu_1h^2n^{v-k}p^{h-1}+O(\mu_1n^{-\delta}) \leq \mu_1/(12e^2h^2\log n).
\end{align}
Let $Y_2$ be a largest hyperedge-disjoint collection of clusters in $H_{n,p}$. We note that $\alpha(G) \leq Y_2$ and use Theorem~\ref{thm:erdtet} to bound $Y_2$:
\begin{align}\label{eq:indset}
\pr\left[\alpha(G)\geq \mu_1/(12h^2\log n)\right]&\leq \pr\left[Y_2\geq \mu_1/(12h^2\log n)\right]\leq \left(\frac{e\mu_212h^2\log n}{\mu_1}\right)^{\mu_2/(12h^2\log n)}\nonumber\\
&\stackrel{(\ref{eq:mu2})}{\leq} e^{-\mu_1/(12h^2\log n)}\stackrel{\eqref{eq:mu1}}{\leq} n^{-2t}.
\end{align}

We now bound $\Delta(G)$. Assume that $\cA$ holds, that is, $H_{n,p}$ does not contain a $(\log n,f)$-cluster for any hyperedge $f$. Fix some hyperedge $f\in E(H_{n,p})$. Let $\cF$ be a collection of clusters $(\alpha_i,F_i)$ such that $f\in E((\alpha_i,F_i))$ for each $i$ and $\alpha_i\neq \alpha_j$ if $i\neq j$.%
\COMMENT{AT - made a few changes to this paragraph. In particular $f\in E((\alpha_i,F_i))$ (instead of $f\in E(F_i)$), makes explaining the bound on $\Delta(G)$ simpler?}
Suppose that $|\cF|\geq h\log n+1$. For each cluster $(\alpha_i,F_i)$ in $\cF$, let $e_i$ be a hyperedge shared by $\alpha_i$ and $F_i$. The $\alpha_i$ are hyperedge-disjoint, so $f\in E(F_i)$ for all but at most one cluster $(\alpha_i,F_i)\in\cF$ where $f\in E(\alpha_i)$. If $\cF$ contains such a cluster, delete it. Then, starting with $i=1$, if $(\alpha_i,F_i)$ has not already been deleted, delete from $\cF$ any clusters $(\alpha_j, F_j)$ with $j>i$ such that $e_j$ lies in $(\alpha_i,F_i)$. Do this for each $i$ in turn. Since the $\alpha_i$ are hyperedge-disjoint, at each step we delete at most $h-1$ clusters from $\cF$. So a collection $\cF'\subseteq \cF$ of at least $\log n$ clusters remains. But the set of all $F_i$ such that $(\alpha_i,F_i)\in \cF'$ contains a $(\log n, f)$-cluster in $H_{n,p}$.
Therefore, $|\cF|< h\log n+1$.
Since every cluster has less than $2h$ hyperedges, we must have
\begin{equation}
\Delta(G)< 2h^2\log n.
\label{eq:maxdeg}
\end{equation}
So, if $\cA$ holds, if $\alpha(G)< \mu_1/(12h^2\log n)$ and if $|Y_1|\geq\mu_1/2$, then
$$|G|\stackrel{\eqref{eq:grapheq}, \eqref{eq:maxdeg}}{\leq} (2h^2\log n+1)\mu_1/(12h^2\log n)\leq \mu_1/4<|Y_1|.$$
Thus,
\begin{equation*}
\pr[(X^-=0) \cap \cA]= \pr[(|G|=Y_1)\cap \cA] \leq\pr[Y_1\leq \mu_1/2]+\pr[\alpha(G)\geq \mu_1/(12h^2\log n)]\stackrel{\eqref{eq:Y1},\eqref{eq:indset}}{\leq} 2n^{-2t},
\end{equation*}
as desired.
\end{proof}

\subsection{Combining the bounds}

We now use Lemmas~\ref{lem:A}~and~\ref{lem:XS} to prove Theorem~\ref{thm:main}.

\begin{proofof}{Theorem~\ref{thm:main}}
Define $\cB$ to be the event that there exist $T\subseteq [n]$ of size $t$ and an ordered sequence $V=(v_1, v_2, \dots, v_{k-1})$
of distinct vertices such that  $v_i \in [n]\setminus T$ for each $1\leq i \leq k-1$ and $X^-=0$. As remarked
before Lemma~\ref{lem:XS}, $\Delta_{k-1}(R_{n,1}) \geq t$ implies $\cB$. So we can apply Lemmas~\ref{lem:A}~and~\ref{lem:XS} to see that
\begin{align*}
\pr[\Delta_{k-1}(R_{n,1})\geq t]&\leq \pr[\cB]\leq \pr[\cA^c]+\pr[\cA\cap \cB]\leq n^{-k}+n^{t+k-1}(2n^{-2t})=o(1).
\end{align*}
This completes the proof of Theorem~\ref{thm:main}.
\end{proofof}

\medskip

{\footnotesize \obeylines \parindent=0pt

Daniela K\"{u}hn, Deryk Osthus, Amelia Taylor 
School of Mathematics
University of Birmingham
Edgbaston
Birmingham
B15 2TT
UK
}
\begin{flushleft}
{\it{E-mail addresses}:
\tt{\{d.kuhn, d.osthus, amt908\}@bham.ac.uk}}
\end{flushleft}

\end{document}